\documentclass[journal]{IEEEtran}

\usepackage{amssymb}
\usepackage{graphicx}

\ifCLASSOPTIONcompsoc
\usepackage[caption=false,font=normalsize,labelfon
t=sf,textfont=sf]{subfig}
\else
\usepackage[caption=false,font=footnotesize]{subfig}
\fi

\usepackage{amsthm}
\usepackage{amsmath}
\usepackage{amsfonts}
\usepackage[usenames,dvipsnames]{pstricks}
\usepackage{epsfig}
\usepackage{url}
\newtheorem{lem}{Lemma}
\newtheorem{prop}{Proposition}
\newcommand{\beq}{\begin{equation}}
\newcommand{\si}{s_{in}}
\renewcommand{\r}{\mathbb{R}_+^2}
\renewcommand{\o}{\overline}
\renewcommand{\u}{\underline}
\newcommand{\eeq}{\end{equation}}

\newcommand{\com}[1]{{ {#1}}}

\begin{document}
\title{Hybrid Control of a Bioreactor with Quantized Measurements: Extended Version}
%
%
%

\author{Francis Mairet  and Jean-Luc Gouz\'e
\thanks{F. Mairet  and J.-L. Gouz\'e are with Inria Biocore, 2004 route des Lucioles, BP 93, 06902 Sophia-Antipolis Cedex, FRANCE. E-mail: francis.mairet@inria.fr, jean-luc.gouze@inria.fr.  \com{A preliminary version of this work appeared in the proceedings of the 1st International Conference on Formal Methods in Macro-Biology, 22-24 Sept. 2014, Noum\'ea (New Caledonia). This work was supported by Phycover (ANR-14-CE04-0011), PEPS BMI-Pectolyse, and Investissement d'avenir Reset projects.}}}

%

\maketitle

\begin{abstract}
We consider the problem of global stabilization of an unstable bioreactor model (e.g. for anaerobic digestion), when the measurements are discrete and in finite number (``quantized''), with  control of the dilution rate. The model is a differential system with two variables, and the output is the biomass growth. The measurements define regions in the state space, and they can be perfect or uncertain (i.e. without or with overlaps). We show that, under appropriate assumptions, a quantized control may lead to global stabilization: trajectories have to follow some transitions between the regions, until the final region where they converge toward the reference equilibrium. On the boundary between regions, the solutions are defined as a Filippov differential inclusion. If the assumptions are not fulfilled, sliding modes may appear, and the transition graphs are not deterministic.
\end{abstract}


%
\IEEEpeerreviewmaketitle


%


\section{Introduction}

Classical control methods are often based on the complete knowledge of some outputs $y (t)$ of the system \cite{sontag1998mathematical}. By complete, we mean that any output $y_i$ is a real number, possibly measured with some noise  $\delta_i$. The control is then built with this (noisy) measurement. These tools have been successfully applied in many domains of science and engineering, e.g. in the domains of biosystems and bioreactors \cite{dochain2010automatic}. However, in these domains, detailed quantitative measurements are often difficult, too expensive, or even impossible. A striking example is the measurements of gene expression by DNA-chips, giving only a Boolean measure equal to on (gene is expressed) or off (not expressed). In the domain of bioprocesses, it frequently happens that only a limited number or level of measurements are available (e.g. low, high, very high ...) because the devices only give a discretized semi-quantitative or qualitative measurement \cite{bernard1999non}. The measure may also be 
quantized by some physical 
device (as the time given by an analogical clock), and give as a result some number among a finite collection.

For this case of quantized outputs, the problem of control has to be considered in a non-classical way: the control cannot be a function of the full continuous state variables anymore, and most likely will change only when the quantized measurement changes. Moreover, the control itself could be quantized, due to physical device limitations.

The above framework has been considered by numerous works, having their own specificity: quantized output and control with adjustable ``zoom'' and different time protocols, \cite{nesic2009unified}, hybrid systems abstracting continuous ones (cf. \cite{lunze2009handbook} for many examples of theories and applications).

 In this paper, we consider a classical problem in the field of bioprocesses: the stabilization of an unstable bioreactor model, representing for example anaerobic digestion,   towards a working set point. Anaerobic digestion is one of the most employed process for (liquid) waste treatment \cite{dochain2010automatic}. Considering a simplified model with two state variables (substrate and biomass), the system has two stable equilibria (and an unstable one, with a separatrix between the two basins of attractions of the respective stable equilibria), one being the (undesirable) washout of the culture (\cite{hess2008design}). The goal is to globally stabilize the process toward the other locally stable reference equilibrium. The (classical) output is the biomass growth (through gaseous production), the control is the dilution  rate (see \cite{steyer2006lessons} for a review of control strategies). \com{There exists many approaches based on well-accepted models (\cite{AM2}), for scalar continuous output \cite{mailleret2004nonlinear,antonelli2003set} or even for delayed and piecewise constant measurements \cite{mazenc2013}}. Some original approaches make use of a supplementary competitor biomass \cite{rapaport2008biological}.

In this paper, we suppose that the outputs are discrete or quantized: there are available in the form of finite discrete measurements. The precise measurements models are described later: roughly, the simplest one is ``perfect'', without noise, meaning that the true value is supposed to be in one of the discrete measurements, and that the transitions between two contiguous discrete measures are perfectly known. The next model is an uncertain model where the discrete measurements may overlap, and the true value is at the intersection between two quantized outputs.  Remark that the model of uncertainty is different from the interval observers approaches (\cite{alcaraz2002software,gouze2000interval})  for the estimation  or regulation \cite{alcaraz2005robust}, where some outputs or kinetics are not well known, but upper or lower bounds are known. Moreover, in the interval observer case, the variables are classical continuous variables.

For this  problem, the general approaches described above do not apply, and we have to turn to more tailored methods, often coming from the theory of hybrid systems, or quantized feedbacks (see above) ...  We here develop our adapted ``hybrid'' approach. It has also some relations with the fuzzy modeling and control approach: see e.g. in a similar bioreactor process the paper \cite{estaben1997fuzzy}. We provide here a more analytic approach, and prove our results of stability with techniques coming from differential inclusions and hybrid systems theory \cite{Filippov}. Our work has some relations with theoretical qualitative control techniques used for piecewise linear systems in the field of genetic regulatory networks (\cite{chaves2011exact}). The approach is also similar to the domain approaches used in hybrid systems theory, where there are some (controlled) transitions between regions, forming a transition graph \cite{belta2006controlling,habets2004control}.

The paper is organized as follows: Section \ref{sec-model} describes the bioreactor model and the measurements models. The next section is devoted to the model analysis in open-loop (with a constant dilution rate). In Section \ref{sec-control}, we propose a control law and show its global stability through the analysis of the transitions between regions (with the help of the Filippov definition of solutions of differential inclusion). Finally, in Section \ref{sec-choose1}, we explain how to choose the dilution rates by a graphical approach, and we end by giving  some simulations, with or without uncertainty.

\section{Framework}\label{sec-model}
\subsection{Model presentation}
In a perfectly mixed continuous reactor, the growth of biomass $x$ limited by a substrate $s$ can be described by the following system (see \cite{bastin1990line,dochain2010automatic}): 
\begin{equation}\label{eq-sys}
\left\lbrace\begin{array}{l}
\dot s= u(t) (\si -s) - k \mu(s)x \\
\dot x = (\mu(s)-u(t)) x
\end{array}\right.
\end{equation}
where $\si$ is the input substrate concentration, $u(t)$ the dilution rate, $k$ the pseudo yield coefficient, and $\mu(s)$ the specific growth rate. 

Given $\xi=(s,x)$, let us rewrite System \eqref{eq-sys} as $\dot \xi=f(\xi,u(t))$, where the dilution rate $u(t)$ is the manipulated input.\\

The specific growth rate $\mu(s)$ is assumed to be a Haldane function (i.e. with substrate inhibition) \cite{AM2}:
\begin{equation}
\label{mus}
\mu(s)=\bar \mu \frac{s}{k_S + s + s^2/k_I}
\end{equation}
where $\bar \mu, k_S, k_I$  are positive parameters. This function admits a maximum for a substrate concentration $s=\sqrt{k_S k_I}:=\bar s$, and we will assume $\bar s<\si$.

 \begin{lem}\label{lem-bound}
 The solutions of System \eqref{eq-sys} with initial conditions in the positive orthant are positive and bounded.
 \end{lem}
\begin{proof}
 It is easy to check that the solutions stay positive.
Now consider $z=s+kx$ whose derivative writes
$$
\dot z=u(t)(\si-z).$$
It follows that $z$ is upper bounded by $\max(z(0),\si)$, and so is $kx$. Finally, if $s(t)> \si$, then $\dot s(t)<0$, therefore $s$ is upper bounded by  $\max(s(0),\si)$.
\end{proof}

In the following, we will assume initial conditions within the interior of the positive orthant.

\subsection{Quantized measurements} \label{measu}

We consider that a proxy of biomass growth $y(\xi)=\alpha \mu(s)x$  is monitored (e.g. through gas production), but in a quantized way, in the form of a more or less qualitative measure:  it can be levels (high, medium, low...) or discrete measures. 
Finally, we only know that  $y(\xi)$ is in a given range, or equivalently that  $\xi$ is in a given region (parameter $\alpha$ is a positive yield coefficient):
\begin{align*}
Y_i=&\{\xi \in \r:\u y_i \leq y(\xi) \leq \o y_i \}, \ i=1,\ldots,n-1,\\
Y_n=&\{\xi \in \r:\u y_n \leq y(\xi) \}.
\end{align*}

where $0=\u y_1<\u y_2<\ldots<\u y_n$ and  $\o y_1<\o y_2<\ldots<\o y_{n-1}$. We will consider two cases:
\begin{itemize}
\item (A1). \textit{Perfect} quantized measurements: 
$$
\o y_i= \u y_{i+1}, \ \forall i\in\{1,\ldots,n-1\}.
$$
This corresponds to the case where there is no overlap between regions. The boundaries are perfectly defined and measured.\\
\psscalebox{1.0 1.0} 

{\begin{center}
\begin{pspicture}(0,-0.8)(8.8,1)
\psline[linecolor=black, linewidth=0.02, arrowsize=0.1cm 2.0,arrowlength=1.9,arrowinset=0.0]{->}(0.0,0)(7.8736844,0)(8.8,0)
\rput[bl](.3,-.7){0}
\rput[bl](2.8,-.8){$\underline y_i=\o y_{i-1}$}
\rput[bl](5.3,-.8){$\o y_i=\u y_{i+1}$}
\rput[bl](1.4,0.3){$\overbrace{\qquad\qquad\quad  }^{Y_{i-1}}$}
\rput[bl](3.2,0.3){$\overbrace{\qquad\qquad\qquad\quad}^{Y_i}$}
\rput[bl](5.7,0.3){$\overbrace{\qquad\qquad\qquad}^{Y_{i+1}}$}
\rput[bl](8.5,-.7){$y$}
\psline[linecolor=black, linewidth=0.02](.4,0.2)(.4,-0.2)
\psline[linecolor=black, linewidth=0.02](1.4,0.2)(1.4,-0.2)
\psline[linecolor=black, linewidth=0.02](3.19,0.2)(3.19,-0.2)
\psline[linecolor=black, linewidth=0.02](5.69,0.2)(5.69,-0.2)
\psline[linecolor=black, linewidth=0.02](7.85,0.2)(7.85,-0.2)
\end{pspicture}\end{center}
}
\item (A2). \textit{Uncertain} quantized measurements: 
$$
\u y_i < \o y_{i-1} < \u y_{i+1}, \ \forall i\in\{2,\ldots,n-1\}.
$$
In this case, we have overlaps between the regions. In these overlaps, the measure is not deterministic, and may be any of the  two values.\\

{\begin{center}
\begin{pspicture}(0,-0.9)(8.8,0.9)
\psline[linecolor=black, linewidth=0.02, arrowsize=0.1cm 2.0,arrowlength=1.9,arrowinset=0.0]{->}(0.0,0)(7.8736844,0)(8.8,0)
\rput[bl](.3,-.7){0}
\rput[bl](3.1,-.8){$\underline y_i$}
\rput[bl](3.5,-.8){$\o y_{i-1}$}
\rput[bl](5.5,-.8){$\u y_{i+1}$}
\rput[bl](6.3,-.8){$\o y_i$}
\rput[bl](1.4,0.3){$\overbrace{\qquad\qquad\qquad  }^{Y_{i-1}}$}
\rput[bl](3.2,0.5){$\overbrace{\qquad\qquad\qquad\qquad\quad}^{Y_i}$}
\rput[bl](5.7,0.3){$\overbrace{\qquad\qquad\qquad}^{Y_{i+1}}$}
\rput[bl](8.5,-.7){$y$}
\psline[linecolor=black, linewidth=0.02](.4,0.2)(.4,-0.2)
\psline[linecolor=black, linewidth=0.02](1.4,0.2)(1.4,-0.2)
\psline[linecolor=black, linewidth=0.02](3.19,0.2)(3.19,-0.2)
\psline[linecolor=black, linewidth=0.02](3.5,0.2)(3.5,-0.2)
\psline[linecolor=black, linewidth=0.02](5.69,0.2)(5.69,-0.2)
\psline[linecolor=black, linewidth=0.02](6.4,0.2)(6.4,-0.2)
\psline[linecolor=black, linewidth=0.02](7.85,0.2)(7.85,-0.2)
\end{pspicture}\end{center}
}

\end{itemize}

For both cases, we define (open) regular domains:
$$\tilde Y_i :=Y_i \setminus (Y_{i-1}\cup Y_{i+1}),$$
and (closed) switching domains:
$$Y_{i\mid i+1} :=Y_{i}\cap Y_{i+1}$$
\com{where the measurement is undetermined, i.e. if $\xi \in Y_{i\mid i+1}$, then either $\xi \in Y_i$ or $\xi \in  Y_{i+1}$.}

For perfect measurements (A1), we have $\tilde Y_i=int Y_i$, and the switching domains $Y_{i\mid i+1}$ correspond to the lines $y(\xi)=\o y_i= \u y_{i+1}$.
For uncertain measurements (A2), the switching domains $Y_{i\mid i+1}$ become the regions $\{\xi \in \r:\u y_{i+1} \leq y(\xi) \leq \o y_i \}$.\\
\com{Unless otherwise specified, we consider in the following uncertain measurements.}

\subsection{Quantized control}
Given the risk of washout, our objective is to design a feedback controller that globally stabilizes System \eqref{eq-sys} towards a set-point. Given that measurements are quantized, the controller should be defined with respect to each region:
\beq\label{eq-u0}
\xi(t) \in Y_i \ \Leftrightarrow \ u(t)=D_i, \quad i=1,...,n.
\eeq
Here $D_i$ is the positive dilution rate in region $i$.
This control scheme leads to discontinuities in the vector fields. Moreover, in the switching domains, the control is undetermined. Thus, solutions of System \eqref{eq-sys} under Control law \eqref{eq-u0} are defined in the sense of Filippov, as the solutions of the differential inclusion \cite{Filippov}:
$$
\dot \xi \in H(\xi)
$$
where $H(\xi)$ is defined on regular domains $\tilde Y_i$ as the ordinary function $H(\xi)=f(\xi,D_i)$, and on switching domains $Y_{i\mid i+1}$  as the closed convex hull of the two vector fields in the two domains $i$ and $i+1$:
$$
H(\xi)=\overline{co}\{f(\xi,D_i),f(\xi,D_{i+1})\}.
$$
Following \cite{CasJonGou06,chaves2011exact}, a solution of System \eqref{eq-sys} under Control law \eqref{eq-u0} on $[0,T]$ is an absolutely continuous (w.r.t. $t$) function $\xi(t, \xi_0)$ such that  $\xi(\com{0}, \xi_0)= \xi_0$ and  $\dot \xi\in H(\xi)$ for almost all $t\in[0,T]$.

\section{Model analysis with a constant dilution}\label{sec-Dcst}

\begin{figure}[h]
\centering
\includegraphics[scale=0.45]{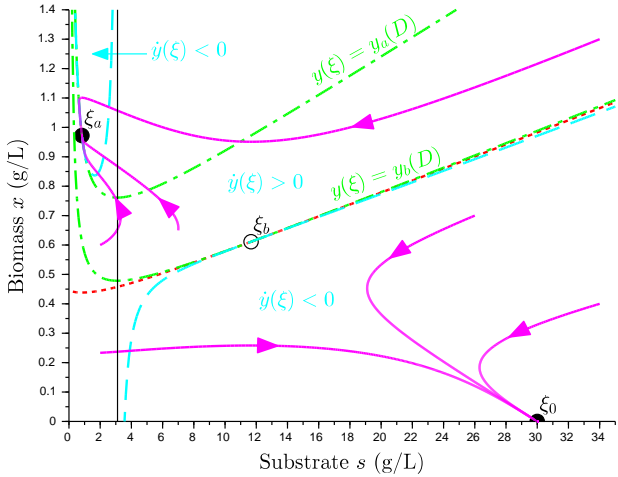} 
\caption{Phase portrait of System \eqref{eq-sys} with a constant dilution rate $u(t)=D\in (\mu(\si),\mu(\bar s))$ (case ii of Proposition \ref{prop-Dcst}). Magenta lines: trajectories, cyan dashed lines: nullcline $\dot y(\xi)=0$, green dash dotted lines: isolines $y(\xi)=y_a(D)$ and $y(\xi)=y_b(D)$, red dotted line: separatrix, black vertical line: $s=\bar s$, dark circles: stable equilibria, open circle: unstable equilibrium.  } 
\label{fig:Dcst}                                 
\end{figure} 

We first consider  System \eqref{eq-sys} when a constant dilution rate $D$ is applied in the whole space (i.e. $u(t)=D,\ \forall t\geq0$). In this case, the system is a classical ordinary differential equation  which can present bistability, with  a risk of washout. 
Let us denote $s_a(D)$ and $s_b(D)$ the two solutions, for $D\in(0,\mu(\bar s))$, of the equation $\mu(s)=D$, with $0<s_a(D)<\bar s < s_b(D)$. For the Haldane growth rate defined by (\ref{mus}), we have:
$$
s_{a}(D)=\frac{k_I}{2}\left(\frac{\bar \mu}{D}-1\right) - \sqrt{\left(\frac{k_I}{2}\left(\frac{\bar \mu}{D}-1\right)\right)^2 - k_S k_I}$$
$$
s_{b}(D)=\frac{k_I}{2}\left(\frac{\bar \mu}{D}-1\right) + \sqrt{\left(\frac{k_I}{2}\left(\frac{\bar \mu}{D}-1\right)\right)^2 - k_S k_I}.$$
The asymptotic behavior of the system can be summarized as follows:

\begin{prop}\label{prop-Dcst} Consider System \eqref{eq-sys} with a constant dilution rate $u(t)=D$  and initial conditions in the interior of the positive orthant.
\begin{itemize}
\item[(i)] If $D<\mu(\si)$, the system admits a globally exponentially stable equilibrium $\xi_a(D)=\left(s_a(D),\frac{\si-s_a(D)}{k}\right)$.
\item[(ii)] If $\mu(\si)<D<\mu(\bar s)$, the system admits two locally exponentially stable equilibria, a working point $\xi_a(D)=\left(s_a(D),\frac{\si-s_a(D)}{k}\right)$ and the washout $\xi_0=(\si,0)$, and a saddle point $\xi_b(D)=\left(s_b(D),\frac{\si-s_b(D)}{k}\right)$, see Figure~\ref{fig:Dcst}.
\item[(iii)]If $D>\mu(\bar s)$, the washout $\xi_0=(\si,0)$ is globally exponentially stable.
\end{itemize}
\end{prop}
\begin{proof}
See \cite{hess2008design}.
\end{proof}

 For $D\in(0,\mu(\bar s))$, let us define:
\begin{align*}
y_j(D)=\frac{\alpha D}{k}[\si-s_j(D)],\quad j=a,b.\\
\end{align*}
 $y_a(D)$ and $y_b(D)$ are the growth proxy obtained respectively at the equilibria $\xi_a(D)$ and $\xi_b(D)$ (if it exists)\footnote{if $D<\mu(\si)$, $\xi_b(D)$ does not exist and $y_b(D)<0$.}.

In order to design our control law, we need to provide some further properties of the system dynamics. In particular, we need to characterize $\dot y(\xi)$, the time derivative of $y(\xi)$ along a trajectory of System \eqref{eq-sys} with a constant dilution rate $u(t)=D$:
$$
\dot y(\xi)= \alpha \left[D (\si -s) - k \mu(s)x\right]\mu'(s)x + \alpha \mu(s) (\mu(s)-D) x. $$

 Let consider the following functions:
\begin{align*}
g_D:&\ s\longmapsto \frac{\mu(s)-D}{k\mu'(s)}+\frac{D(\si-s)}{k\mu(s)}\\
h^j_D:&\ s\longmapsto \frac{D(\si-s_j(D))}{k\mu(s)}, \quad j=a,b,
\end{align*}
\com{defined respectively on $(0,\bar s)\cup(\bar s,+\infty)$ and $(0,+\infty)$.}

In the $(s,x)$ plane, $g_D(s)$, $h^a_D(s)$ and $h^b_D(s)$ represent respectively the nullcline $\dot y(\xi)=0$ and the isolines $y(\xi)=y_a(D)$ and $y(\xi)=y_b(D)$ (i.e. passing through the equilibria $\xi_a(D)$ and $\xi_b(D)$), see Figure \ref{fig:Dcst}. Knowing that the nullcline $\dot y(\xi)=0$ is tangent to the isoline $y(\xi)=y_a(D)$ (resp. $y(\xi)=y_b(D)$) at the equilibrium point $\xi_a(D)$ (resp. $\xi_b(D)$),  we will determine in the next lemma the relative positions of these curves, see Fig. \ref{fig:Dcst}.

\begin{lem}\label{lem-y}
 \com{Consider System \eqref{eq-sys} with a constant dilution rate $u(t)=D$.} 
 \begin{itemize}
 \item[(i)]For $s\in(0,\bar s)$, we have $g_D(s)\geq h^a_D(s)$:  the nullcline $\dot y(\xi)=0$ is above the isoline $y(\xi)=y_a(D)$.
  \item[(ii)] For $s\in(\bar s,s_{in})$, we have $g_D(s)\leq h^b_D(s)$: the nullcline $\dot y(\xi)=0$ is below the isoline $y(\xi)=y_b(D)$.
 \end{itemize}
\end{lem}
\begin{proof}
See Appendix.
\end{proof}
 
 This allows us to determine the monotonicity of $y(\xi)$ in a region of interest (for the design of the control law). 
 
 \begin{lem}\label{lem-doty}
  Consider System \eqref{eq-sys} with a constant dilution rate $u(t)=D$. 
 For $\xi\in \r$ such that $y_b(D) <y(\xi)<y_a(D)$, we have $\dot y(\xi)>0$.
\end{lem}
\begin{proof}
See Appendix.
\end{proof}

\section{Control with quantized measurements}\label{sec-control}
\subsection{Control design}
\com{Our goal is to globally stabilize the system towards an equilibrium with a high productivity (the productivity is the output $\alpha \mu(s)x$), corresponding to a high dilution rate $D_n$ where there is bistability in open loop (case (ii) of Proposition \ref{prop-Dcst}).} Number $n$ is the number of measurements (see section \ref{measu}) and the control is such that $D_1 < D_2< \ldots < D_n$. We consider the following control law, based on the quantized measurements $y(\xi)$, and constant within a given region:
\beq\label{eq-u}
\forall t \geq 0, \quad \xi(t) \in Y_i \ \Leftrightarrow  u(\com{t})=D_i, 
\eeq
given that the following conditions are fulfilled:
\begin{align}
& y_b(D_i)<\u y_i \quad  i=1,\ldots,n, \label{eq-condi1} \\ 
& y_a(D_i)>\o y_i \quad i=1,\ldots,n-1,\label{eq-condi2} \\
&y_a(D_n)>\o y_{n-1}. \label{eq-condi3}
\end{align}

These conditions make the equilibrium $\xi_a(D_n)$ globally stable, as we will see below. In Section \ref{sec-choose}, we will precise  how to choose the $D_i$ such that these conditions hold.\\
In order to prove the asymptotic behavior of System \eqref{eq-sys} under Control law (\ref{eq-u}-\ref{eq-condi3}), the study will be divided into three steps:
\begin{itemize}
\item the dynamics in one region,
\item the transition between two regions,
\item the  global dynamics.
\end{itemize}
This approach is similar to those deducing the global dynamics from a ``transition graph'' of possible transitions between regions \cite{belta2006controlling}.

\subsection{Dynamics in one region with a given dilution: exit of domain}
We first focus on  a region $Y_i, \ i<n$. A constant dilution $D_i$ - such that Conditions (\ref{eq-condi1}-\ref{eq-condi2}) for $i$ hold - is applied. These conditions
 guarantee that the stable operating equilibrium for this dilution  (see  Proposition \ref{prop-Dcst}) is located in an upper region $Y_j$, $j>i$, while the saddle point is located in a lower region $Y_k$, $k<i$. This allows us to establish the following lemma:
 
 \begin{lem}\label{lem-1region}
For any $i\in\{1,...,n-1\}$, consider System \eqref{eq-sys} under a constant control $u(t)=D_i$, such that Conditions (\ref{eq-condi1}-\ref{eq-condi2}) for $i$ hold.  All solutions with initial conditions in $Y_i$ leaves this domain, crossing the boundary $ y(\xi) = \o y_i $.
\end{lem}

\begin{proof}
Let us consider the function $V(\xi)=y_a(D_n)-y(\xi)$ on $Y_i$.
Given Conditions (\ref{eq-condi1}-\ref{eq-condi2}), we get $\forall \xi \in Y_i$:
  $$y_b(D_i)<\u y_i<y(\xi)<\o y_i<y_a(D_i).$$
   Since a constant dilution rate $D_i$ is applied, we can apply Lemma \ref{lem-doty} to conclude that $\dot y(\xi)>0$.\\
Thus, $V(\xi)$ is decreasing on $Y_i$. 
Recalling that the trajectories are also bounded (Lemma \ref{lem-bound}), we can apply LaSalle invariance theorem \cite{lasalle1976stability} on the domain $\Omega_1:=\{\xi\in Y_i\mid x\leq \max(z(0),\si)/k,\ s\leq \max(s(0),\si) \}$. Given that the set of all the points in $\Omega_1$ where $\dot V(\xi)=0$ is empty,  any trajectory starting in $\Omega_1$ will leave this region. The boundaries $x=\max(z(0),\si)/k $ and $s=\max(s(0),\si)$ are repulsive (see Proof of Lemma \ref{lem-bound}). Finally, the boundary $y(\xi)=\u y_i$ corresponds to the maximum of $V(\xi)$ on $\Omega_1$, so every trajectory will leave this domain, crossing the boundary $ y(\xi) = \o y_i $.
\end{proof}

\subsection{Transition between two regions}

Now we will characterize the transition between regions (as we have seen above, the intersection can be either a simple curve in the case of perfect measurements, or a region with non empty interior in the uncertain case):

 \begin{lem}\label{lem-transition}
For any $i\in\{1,...,n-1\}$, consider System \eqref{eq-sys} under Control law \eqref{eq-u} with Conditions (\ref{eq-condi1}-\ref{eq-condi2}) for $i, i+1$\footnote{or if $i=n-1$, Conditions (\ref{eq-condi1}) for $n-1,n$, Condition (\ref{eq-condi2}) for $n-1$, and  Condition (\ref{eq-condi3}).}.  All trajectories with initial conditions in $ \tilde Y_i \cup Y_{i\mid i+1}$ enter the regular domain $\tilde Y_{i+1}$.
\end{lem}

\begin{proof}
 First, we consider $i\neq n-1$. We will follow the same reasoning as for the previous lemma, applying LaSalle theorem on a domain 
 \begin{align*}
 \Omega_2:=\{\xi\in int\r & \mid x\leq \max(z(0),\si)/k,\\& s\leq \max(s(0),\si),\ \o y_{i-1}<y(\xi)<y^\dag \}
 \end{align*}
\com{for any $y^\dag\in \tilde Y_{i+1}$.} We can show that the functional $V(\xi)=y_a(D_n)-y(\xi)$  is decreasing on $Y_{i}$ whenever $u=D_i$. Similarly, $V(\xi)$  is also decreasing on $Y_{i+1}$ whenever $u=D_{i+1}$.
Now under Control law \eqref{eq-u}, we have shown that $V(\xi)$  is decreasing on the regular domains $\tilde Y_{i}$ and $\tilde Y_{i+1}$. $V(\xi)$ is a regular $C^1$ function, and can be differentiated along the differential inclusion. On the switching domains $Y_{i\mid i+1}$, we have:
$$
\dot V(\xi)\in \overline{co}\left\lbrace\frac{\partial V}{\partial x} f(\xi,D_i),\frac{\partial V}{\partial x}f(\xi,D_{i+1})\right\rbrace <0.
$$
Thus, $V(\xi)$ is decreasing on $\Omega_2$. Following the proof of Lemma \ref{lem-1region} concerning the boundaries, we can deduce that every trajectory will reach the boundary $y(\xi)=y^\dag$, i.e. it will enter  $\tilde Y_{i+1}$.\\
For $i=n-1$, taking \com{any $y^\dag\in(\o y_{n-1},y_a(D_n))$}, we can show similarly that  $V(\xi)$ is decreasing on $\Omega_2$ so every trajectory will enter the region $\tilde Y_n$.
\end{proof}

Following the same proof, we can show that the reverse path is not possible, in particular for the last region:

 \begin{lem}\label{lem-transition-inv}
Consider System \eqref{eq-sys} under Control law \eqref{eq-u} with   Conditions (\ref{eq-condi1}) for $n-1,n$, Condition (\ref{eq-condi2}) for $n-1$, and  Condition (\ref{eq-condi3}).  The regular domain $\tilde Y_{n}$ is positively invariant.
\end{lem}

\subsection{Global dynamics}
\com{Now we are in a position to present the main result of the paper:}
\begin{prop} \label{Propo}
Control law (\ref{eq-u}) \com{under Conditions (\ref{eq-condi1}-\ref{eq-condi3}) with perfect or uncertain measurements (A1 or A2)} globally stabilizes System \eqref{eq-sys} towards the point $\xi_a(D_n)$.
\end{prop}
\begin{proof}
From Lemmas \ref{lem-transition} and \ref{lem-transition-inv}, we can deduce that every trajectory will enter the regular domain $\tilde Y_n$, and that this domain is positively invariant.

System \eqref{eq-sys} under a constant control $u(t)=D_n$ has two non-trivial equilibria (see Proposition \ref{prop-Dcst}): $\xi_a(D_n)$, and $\xi_b(D_n)$.
The growth proxy at these two points satisfy $y_a(D_n)>\o y_{n-1}$ and $y_b(D_n)<\u y_n$  (Conditions (\ref{eq-condi1},\ref{eq-condi3})), so there is only one equilibrium in $\tilde Y_n$: $\xi_a(D_n)$. Moreover, it is easy to check that $\tilde Y_n$ is in  the basin of attraction of $\xi_a(D_n)$, therefore all  trajectories will converge toward this equilibrium.
 
\end{proof}

\section{Implementation of the control law}\label{sec-choose1}

\begin{figure*}[!ht]
\centering
\includegraphics[trim = 5mm 00mm 20mm 0mm, clip,scale=.8]{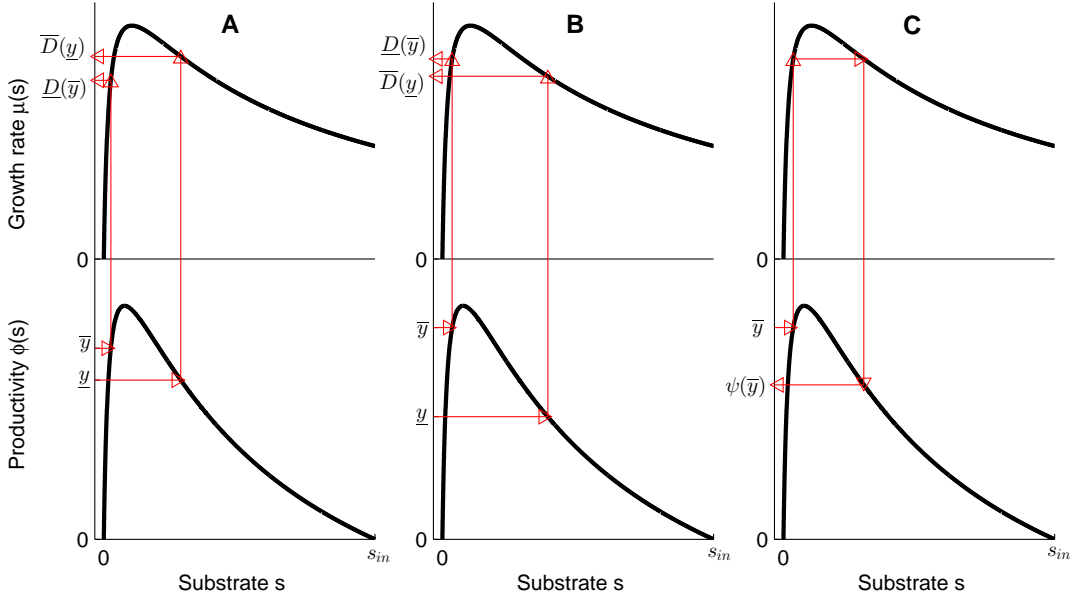} 
\caption{How to choose $D$ in one region (see Section \ref{sec-choose}). A: One can choose any $D\in (\u D(\o y), \o D(\u y))$; B: it is impossible to fulfill Conditions (\ref{eq-condi1}-\ref{eq-condi2}), given that $\u D(\o y)> \o D(\u y))$; C: $\psi(\o y)$ represents the lower bound for $\u y$ (limit case).} 
\label{fig:tuning1}                                
\end{figure*} 

\subsection{How to fulfill Conditions (\ref{eq-condi1}-\ref{eq-condi3})}\label{sec-choose}
The global stability of the control law is based on  Conditions (\ref{eq-condi1}-\ref{eq-condi3}). We now wonder how to easily check if these conditions  hold, or how to choose the dilution rates and/or to define the regions in order to fulfill these conditions. In this purpose, a graphical approach can be used.

As an example, we will consider the case where the regions are imposed (by technical constraints) and we want to find the different dilution rates $D_i$ such that Conditions (\ref{eq-condi1}-\ref{eq-condi3}) hold.\\
Our objective is to globally stabilize the equilibrium point $\xi_a(D^*)\in \tilde Y_n$, with $\mu(\si)<D^*<\mu(s^\diamond)$ ($s^\diamond$ is defined just after).\\

Let $\phi(s):=\frac{\alpha}{k}\mu(s)(\si-s)$, which represents the steady state productivity. On $[0,\si]$, $\phi(s)$ admits a maximum for 
$$s^\diamond :=\frac{\si}{1+\sqrt{1+\frac{\si}{k_S}\left(1+\frac{\si}{k_I}\right)}} <\bar s.$$
Note that we impose $D^*<\mu(s^\diamond)$  given that for any $D^*>\mu(s^\diamond)$, the same productivity can be achieved with a smaller dilution rate, leading to a reduced risk of instability. \\

 Let us denote $s_c(y)$ and $s_d(y)$ the two solutions, for $y\in(0,\phi(s^\diamond))$, of the equation $\phi(s)=y$, with $0<s_c(y)<s^\diamond < s_d(y)<\si$.\\
For the Haldane growth rate, we have:
$$
s_{c,d}(y)=\frac{\bar \mu \si - \frac{ky}{\alpha}  \mp \sqrt{\left(\frac{ky}{\alpha}-\bar\mu \si\right)^2-4 k_S\frac{ky}{\alpha}(\frac{ky}{\alpha k_I} +\bar \mu)}}{2\left(\frac{ky}{\alpha k_I}+\bar\mu \right)}.
$$\\

\begin{figure}[!h]
\begin{center}
\includegraphics[trim = 1mm 00mm 00mm 0mm, clip,scale=.7]{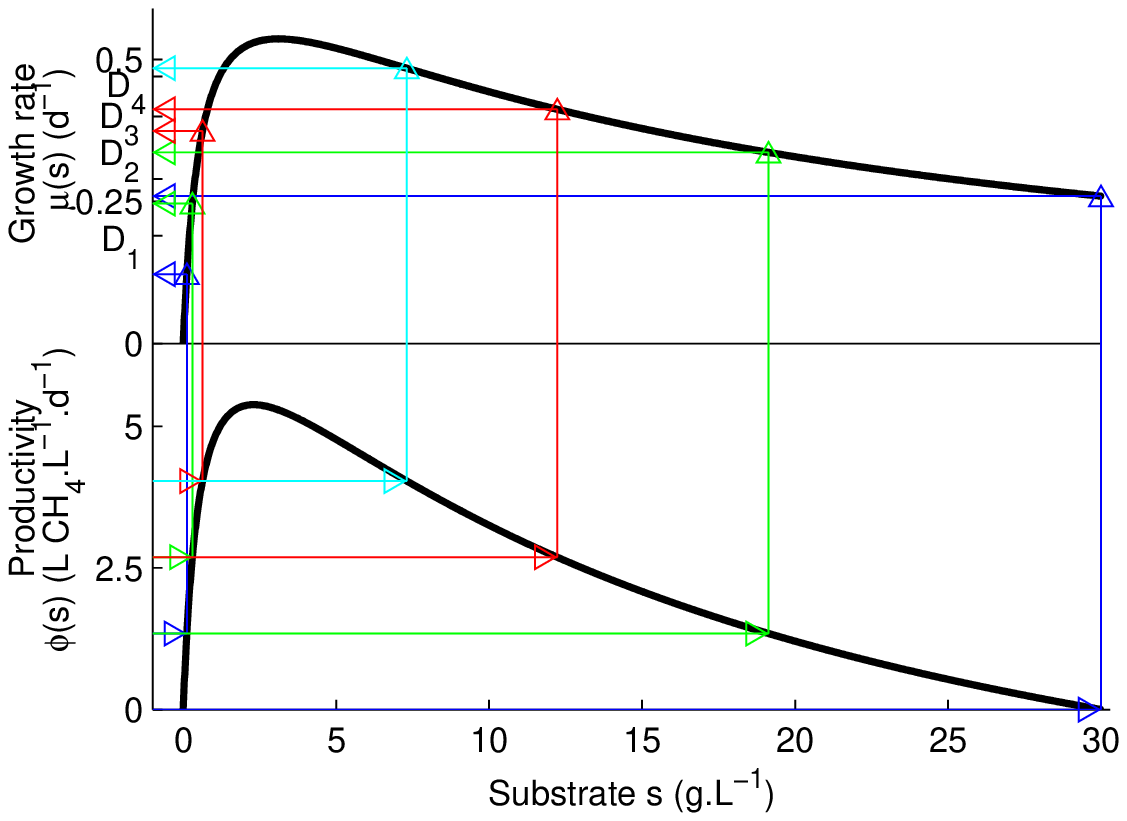} 
\includegraphics[trim = 1mm 00mm 00mm 0mm, clip,scale=.7]{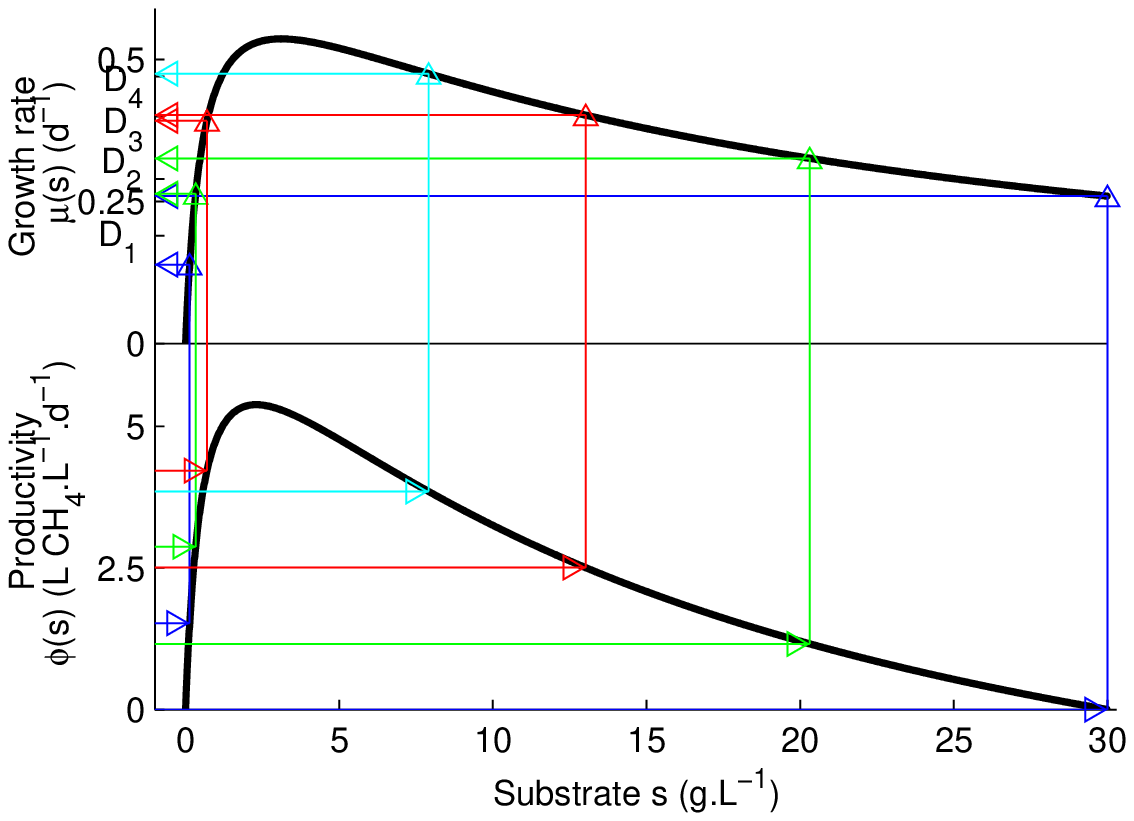} 
\caption{How to choose all the $D_i$ (see Section \ref{sec-choose}). Each color represents a region (1: blue, 2: green, 3: red, 4: cyan). Top: perfect measurements (A1); Bottom: uncertain measurements (A2).} 
\label{fig:tuning2}                                
\end{center}  
\end{figure}

\subsubsection*{How to choose $D_i$ in one region}
For given lower and upper bounds $\u y_i<\o y_i<\phi(s^\diamond)$, we define:
 $$\o D(\u y_i):=\min(\mu(s^\diamond),(\mu\circ  s_d)(\u y_i)) \quad \textrm{and} \quad \u D(\o y_i):=(\mu \circ s_c)(\o y_i).
 $$ 
 
 This can be done analytically or graphically, as shown on Figure \ref{fig:tuning1}A.
 Whenever we choose $D_i\in (\u D(\o y_i), \o D(\u y_i))$, we have
 \begin{itemize}
 \item $s^\diamond<s_d(\u y_i)<s_b(D_i)$, so $y_b(D_i)<\u y_i$ (Condition \eqref{eq-condi1}),
 \item $s_c(\o y_i)<s_a(D_i)<s^\diamond$, so $y_a(D_i)>\o y_i$ (Condition \eqref{eq-condi2}).
 \end{itemize}
so Conditions (\ref{eq-condi1}-\ref{eq-condi2}) for $i$ hold.

 If $\u D(\o y_i)\geq \o D(\u y_i)$ (see Figure \ref{fig:tuning1}B), it is not possible to fulfill the conditions and thus to implement the control law with this measurement range.\\

\subsubsection*{How to choose all the $D_i$}
The procedure proposed in the previous subsection should be repeated for all the regions. We can depict two particular cases:
\begin{itemize}
\item For the first region $Y_1$, given $\u y_1=0$, we actually impose $D_1<\mu(\si)$.
\item For the last region $Y_n$, one should check that $\u D(\o y_{n-1})<D_n=D^*<\o D(\u y_n)$. 
\end{itemize}
This approach is illustrated on Figure \ref{fig:tuning2}.\\

\subsubsection*{Increasing measurement resolution}
We have seen that for a given region, it is not always possible to fulfill Conditions (\ref{eq-condi1}-\ref{eq-condi3}). This gives rise to a question: is there any constraint on the measurement that guarantees the possibility to implement the control law?

 For perfect measurements (A1) with equidistribution, we will show that it will always be possible to implement the control law increasing the number of regions. \\
 
\com{First, we can arbitrarily define the lower bound of the last region: $\u y_n \in (y_b(D^*),y_a(D^*))$, recalling nonetheless that $n$ is unknown}. 
 
 Now, we will determine the limit range of a measurement region. Let  $\hat y$ such that $\u D(\hat y)=\mu(\si)$. For $\u y_i<\o y_i<\hat y$, we have $\u D(\o y_i)<\mu(\si)< \o D(\u y_i)$ so it is always possible to choose a $D_i$ in order to fulfill Conditions (\ref{eq-condi1}-\ref{eq-condi2}).

For $\o y \in [\hat y,\u y_n]$, we define (see Figure \ref{fig:tuning1}C):
 $$\psi(\o y):=(\phi \circ s_b \circ \mu \circ s_c)(\o y).$$
 It gives a lower bound on $\u y_i$: whenever one choose $\u y_i>\psi(\o y_i)$, then we have:
$$\o D(\u y_i)>\o D(\psi(\o y_i))=\u D(\o y_i),$$
  so one can choose any $D_i\in \left((\u D(\o y_i),\o D(\u y_i)\right)$ and Conditions (\ref{eq-condi1}-\ref{eq-condi2}) for $i$ will hold.

We can now define the mapping $\Delta:\o y \mapsto \o y- \psi(\o y)$ on $ [\hat y,\u y_n]$. $\Delta(\o y)$ defines the maximal range of the region with upper bound $\o y$ (in order to be able to find  a dilution rate such that Conditions (\ref{eq-condi1}-\ref{eq-condi2}) hold). $\Delta$ admits a minimum:

$$\Delta_m:=\min_{\o y \in [\hat y,\u y_n]} \Delta(\o y)>0.$$

Thus, for $n>\frac{\u y_n}{\Delta_m}+1$, 
the regions $Y_i$ defined by:
$$
\o y_i= \u y_{i+1}= \frac{i}{n-1} \u y_n, \quad i=1,...,n-1
$$
allow the implementation of the control law.
In conclusions, \com{whenever the measurement resolution is good enough (i.e. the number of regions is high enough),  it is always possible to find a set of dilution rates $D_i$ such that Conditions (\ref{eq-condi1}-\ref{eq-condi3})  hold in the case of perfect measurements (A1) with equidistribution. For uncertain measurements (A2), the same result holds, but the proof is omitted for sake of brevity.}

\begin{figure}[h]   
\centering
\includegraphics[trim = .9cm 0cm 1.3cm .1cm, clip,scale=0.48]{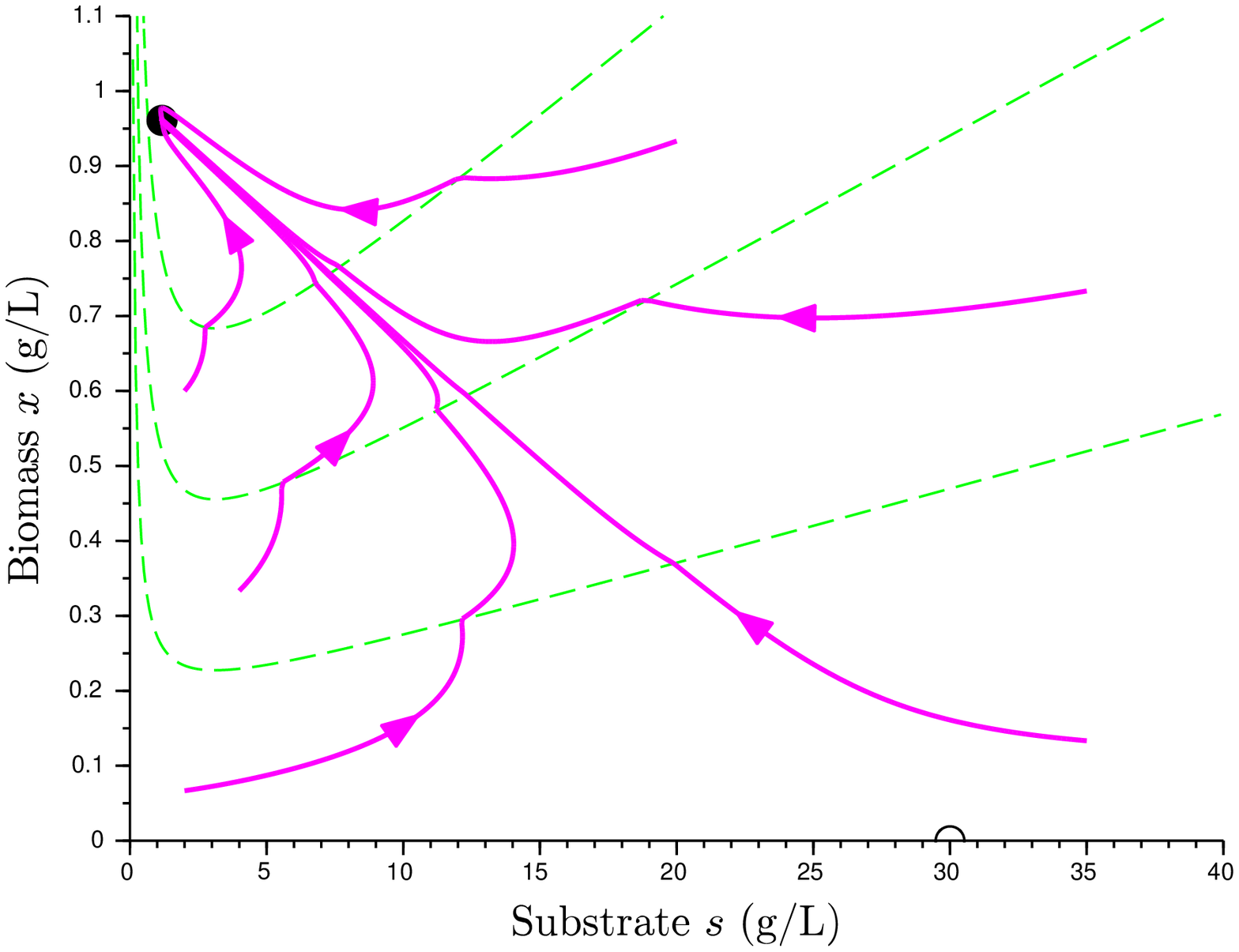} 
\includegraphics[trim = .9cm 0cm 1.3cm .1cm, clip,scale=0.48]{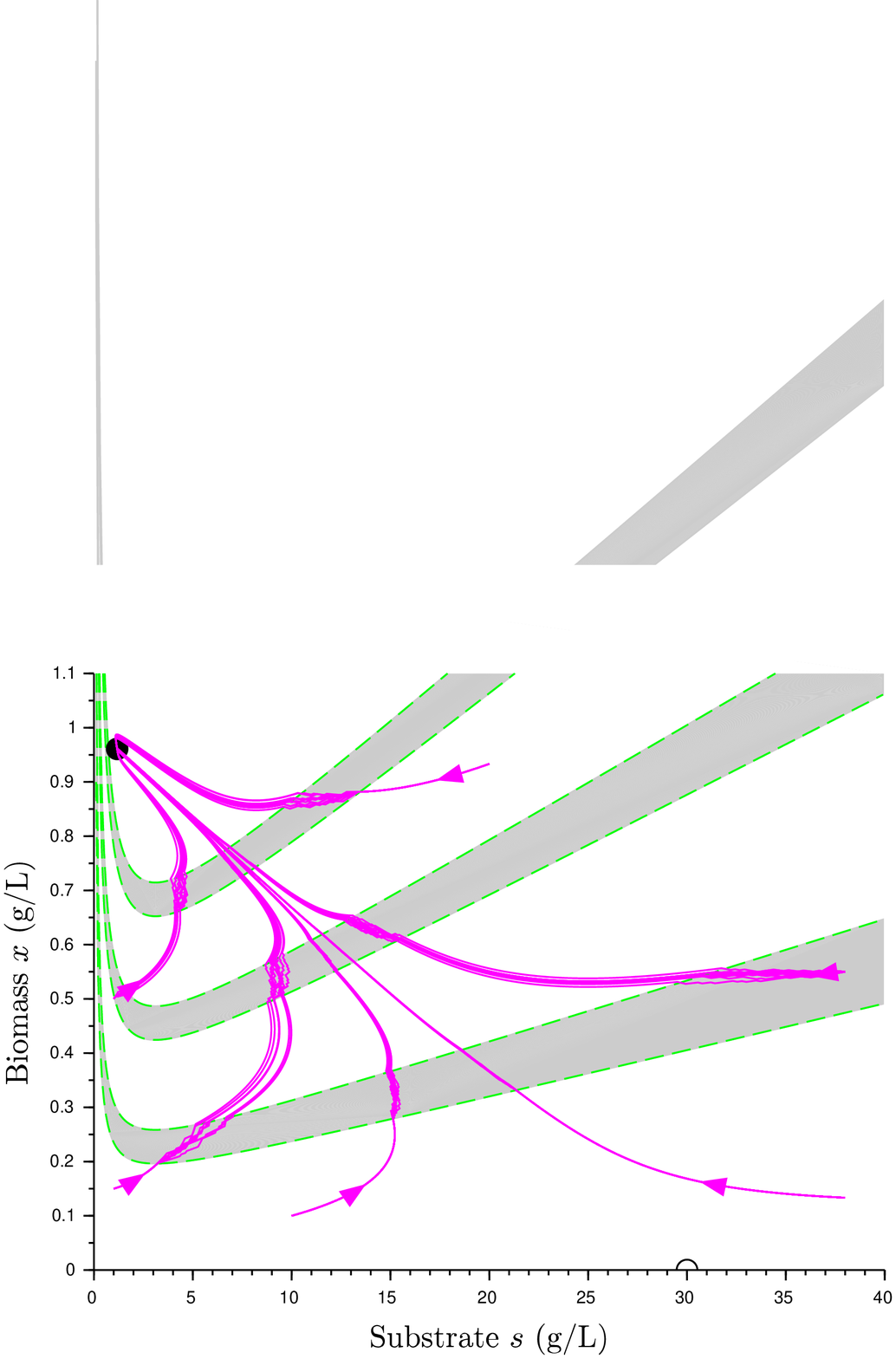}                           
\caption{Trajectories (magenta lines) with Control law \eqref{eq-u} for various initial conditions in the phase portrait. Conditions (\ref{eq-condi1}-\ref{eq-condi3}) are fulfilled, so all the trajectories converge towards the set-point (dark circle), see Proposition \ref{Propo}. Open circle: washout. Top: perfect measurements (A1); Bottom: uncertain measurements (A2). The frontiers are represented by the green dashed lines, switching regions are colored in gray.  In these gray regions, the system is not deterministic. } 
\label{fig:simu1} 
\end{figure}

\subsection{Simulations}

\begin{figure*}[ht]
\centering
\subfloat[Conditions (\ref{eq-condi1}-\ref{eq-condi3}) are fulfilled (see Fig. \ref{fig:simu1})]{\includegraphics[trim = -1cm 0cm -1cm .0cm, clip,scale=1]{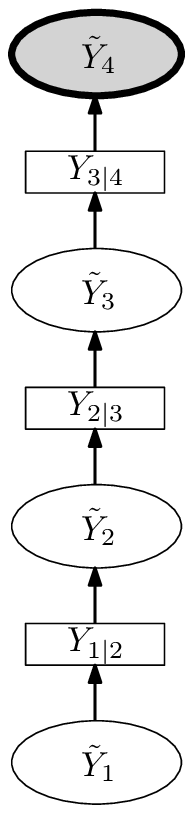} 
\label{fig_first_case}}
\hfil
\subfloat[Condition \eqref{eq-condi1}-2 does not hold (see Fig. \ref{figsimu2})]{\includegraphics[trim = -1cm 0cm -1cm .0cm, clip,scale=1]{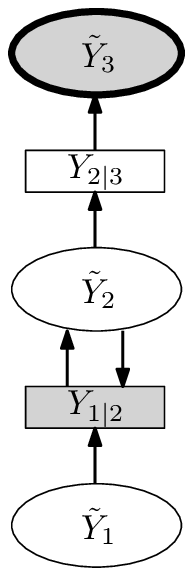} 
\label{fig_second_case}}
\hfil
\subfloat[Condition \eqref{eq-condi2}-2 does not hold (see Fig. \ref{figsimu3})]{\includegraphics[trim = -1cm 0cm -1cm .0cm, clip,scale=1]{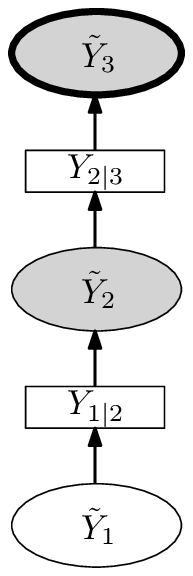} 
\label{fig_third_case}}
\caption{Transitions between regions. Ellipses: regular domains; rectangles: switching domains. White regions are transient, while grey regions have a stable equilibrium. Ellipses with a thick line are positively invariant.}
\label{fig_schema}
\end{figure*}

As an example, we consider the anaerobic digestion process, where the methane production rate is measured. Parameters, given in Table \ref{tab:param}, are inspired from \cite{AM2} (considering only the methanogenesis step). For uncertain measurements, we use discrete time simulation. At each time step $t_k$ (with $\Delta t=0.05$d), when $\xi(t_k)$ is in a switching region $Y_{i\mid i+1}$, we choose randomly the control $u(t_k)$ between $D_i$ and $D_{i+1}$. In this case, we perform various simulations for a same initial condition. 
 \begin{table}[h]
\begin{center}
\caption{Parameter values used for simulation.}\label{tab:param}
\begin{tabular}{lc}
Parameter & Value \\ \hline
 $\bar \mu  $ & 0.74 d$^{-1}$ \\
 $k_S$  & 0.59 g.L$^{-1}$\\ 
  $k_I$  & 16.4 g.L$^{-1}$\\
   $k$  & 30 \\
    $\alpha$  & 11 L CH$_4$.g$^{-1}$\\
     $\si$  & 30 g.L$^{-1}$\\
 \hline
 \end{tabular}
\end{center}
\end{table}

Our objective is to stabilize the equilibrium $\xi_a(D^*)$, with $D^*=0.47\ \textrm{d}^{-1}$  (which corresponds to a productivity of 92\% of the maximal productivity).
We first consider the following perfect measurement set (with equidistant region):
$$
\o y_i= \u y_{i+1}= \frac{i}{n-1}  \o y_n, \quad i=1,...,n-1, 
$$
with $\o y_n= 4 \ \textrm{L CH}_4.\textrm{L}^{-1}.\textrm{d}^{-1}$.
With four regions ($n=4$), we can define a set of dilution rates such that Conditions (\ref{eq-condi1}-\ref{eq-condi3}) are fulfilled (see Fig. \ref{fig:tuning2}):
$$
\begin{array}{ll}
D_1=0.19\ \textrm{d}^{-1}, &  D_2=0.29\ \textrm{d}^{-1}, \\
 D_3=0.4\ \textrm{d}^{-1}, & D_4=0.47\ \textrm{d}^{-1}.
\end{array}
$$

For uncertain measurements, we increased each upper bound and decreased each lower bound by  10\%. It appears that the same dilution rates can be chosen. 

Trajectories for various initial conditions are represented in the phase portrait for perfect and uncertain measurements, see Fig. \ref{fig:simu1}. In accordance with Proposition \ref{Propo}, all the trajectories converge towards the set-point. Thus, the transition graph is deterministic (there is only one transition from a region to the upper one), see Fig. \ref{fig_first_case}.

If the number of regions is reduced (three regions only), it is not possible in this example to choose dilution rates such that Conditions (\ref{eq-condi1}-\ref{eq-condi3}) hold for all $i$.  In this case, some trajectories do not converge towards the set-point. Some regions have transitions towards the upper region, but also towards the lower one. There are sliding modes. This aspect will be further discussed in the next subsection.

\begin{figure}[h]
\centering
\includegraphics[trim = .9cm 0cm 1.3cm .1cm, clip,scale=0.48]{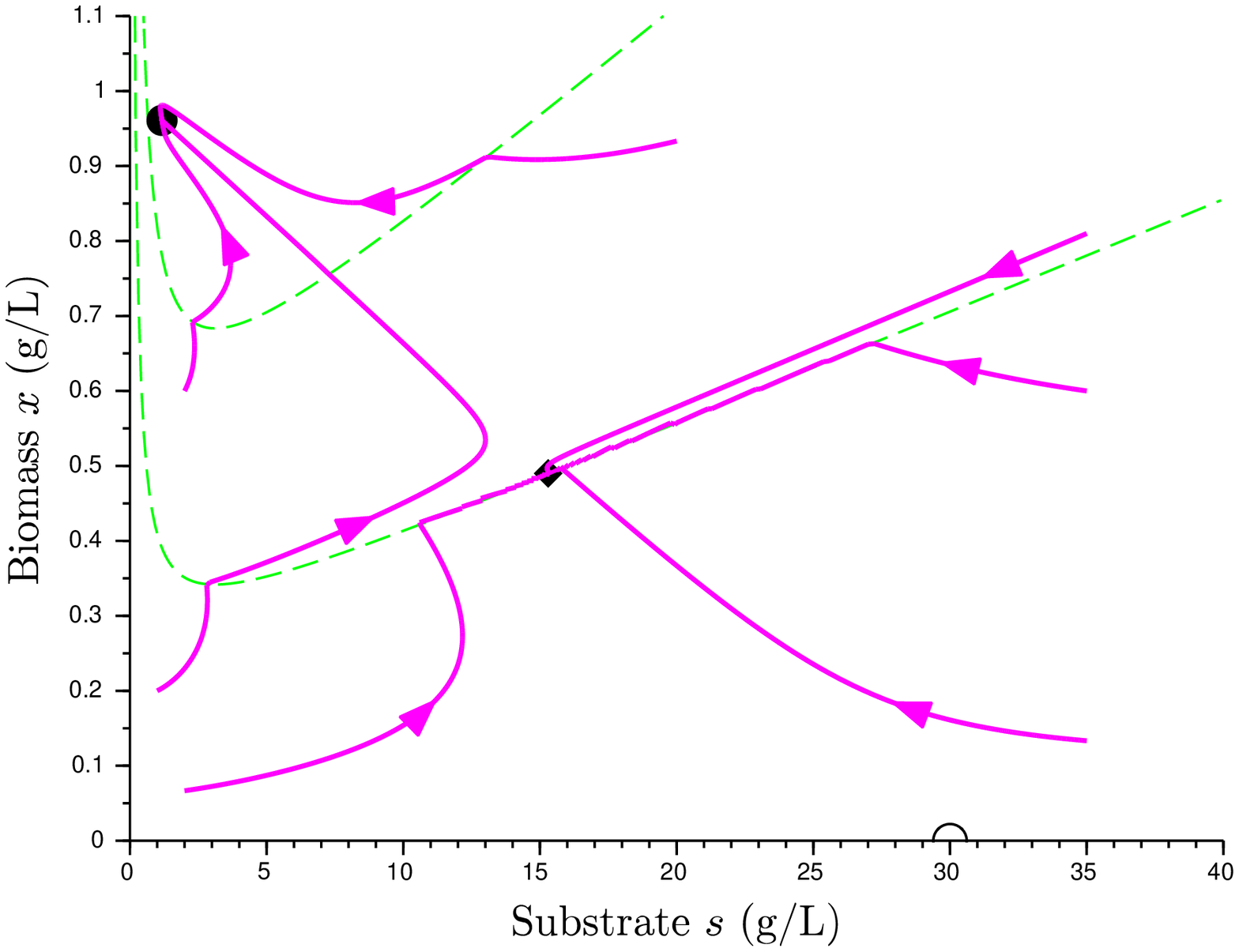} 
\includegraphics[trim = .9cm 0cm 1.3cm .1cm, clip,scale=0.48]{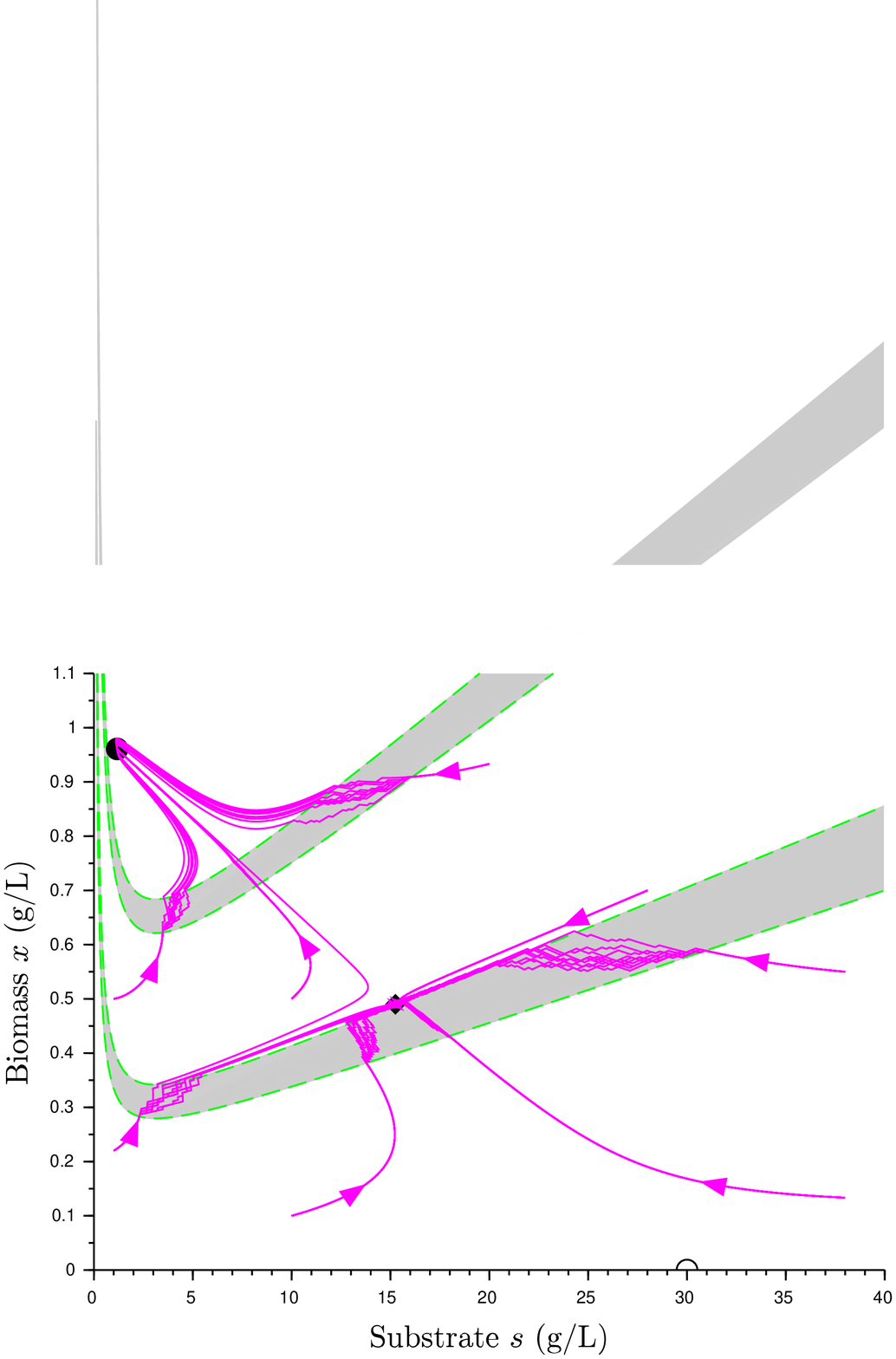} 
\caption{Trajectories with Control law \eqref{eq-u} when Conditions (\ref{eq-condi2})-2 are not fulfilled . Same legend as Fig. \ref{fig:simu1}. Some trajectories converge towards a singular equilibrium point (black diamond) with a sliding mode (see also the associated transition graph in Fig. \ref{fig_second_case}). } 
\label{figsimu2}      
\end{figure}

\subsection{When conditions are not verified: risk of failure}

We here detail   what happens if Conditions (\ref{eq-condi1}-\ref{eq-condi3}) are not fulfilled, and in particular if there is a risk of washout. This point is illustrated by Fig. \ref{figsimu2} and Fig. \ref{figsimu3}.\\

 First, given the previous analysis of the system, one can easily see that only the condition $ y_b(D_1)<\u y_1$, i.e. $D_1<\mu(\si)$ is necessary to prevent a washout, so $D_1$ can be chosen with a safety margin in order to avoid such situation.
Now, if Condition \eqref{eq-condi1} is not fulfilled for some $i>1$, the unstable equilibrium $\xi_b(D_i)$ will be located in the region $Y_i$. Thus, the region $\tilde Y_i$ have transitions towards the lower region, and a trajectory can stay in the switching domain $Y_{i-1|i}$. Given that \com{$z=s+kx$ converges towards $\si$ (cf. proof of Lemma \ref{lem-bound})}, such trajectory will converge towards the intersection between the switching domain and the invariant manifold $z=\si$ (see Fig. \ref{fig_second_case} and Fig. \ref{figsimu2}):
\begin{itemize}
\item For perfect measurements (A1), this gives rise to a sliding mode and the convergence towards a singular equilibrium point.
\item For uncertain measurements (A2), all the  trajectories converge towards a line segment. In our simulation, they actually also converge towards a singular equilibrium point.
\end{itemize}
\com{This situation can be detected by the incessant switches between two regions. In such case, the dilution rate $D_i$ should be slightly decreased.}\\

On the other hand, if Condition \eqref{eq-condi2} does not hold for some $i$, the stable equilibrium $\xi_a(D_i)$ will be located in the region $Y_i$, so some trajectories can converge towards this point instead of going to the next region, see Fig. \ref{fig_third_case} and Fig. \ref{figsimu3}. \com{To detect such situation, a mean escape time for each region can be estimated using model simulations. A trajectory which stay much more than the escape time in one region should have reached an undesirable equilibrium. In this case, the dilution rate $D_i$ should be slightly increased.} 

In all the cases, the trajectories converge towards a point or a line segment. Although it is not desired, this behavior is particularly safe (given that there is theoretically no risk of washout). Moreover, \com{as explained above}, these situations can easily be detected and the dilution rates can be changed accordingly (manually or through a supervision algorithm).

\begin{figure}[h]
\centering
\includegraphics[trim = .9cm 0cm 1.3cm .1cm, clip,scale=0.48]{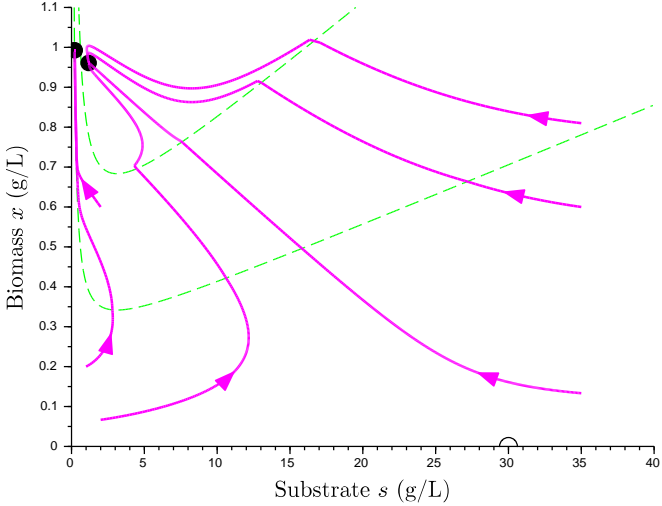} 
\caption{Trajectories with Control law \eqref{eq-u} and perfect measurement when Conditions (\ref{eq-condi1})-2 is not fulfilled. Same legend as Fig. \ref{fig:simu1}. Some trajectories (left of the figure) converge towards an equilibrium point in the region $\tilde Y_2$ (see also the associated transition graph in Fig. \ref{fig_third_case}). }   \label{figsimu3}          
\end{figure}

 \section{Conclusion}
 Given the quantized measurements, we were able to design (under some conditions) a control based on regions and transition between regions. These tools are similar to the ones of piecewise linear systems, and it is possible to draw a transition graph showing all the possible transitions. Moreover, we have seen that for some \com{undesirable} cases, singular behaviors (sliding modes) are possible on the boundaries between regions. We think that this kind of control on domains, and the design of the resulting transition graph, is a promising approach, that we want to deepen in future works. This approach could be generalized to other classical systems, e.g. in mathematical ecology.

 \bibliographystyle{splncs03}

\bibliography{hybridbioreactor}   

\section{Appendix}
\subsection*{Proof of Lemma \ref{lem-y}}
Let us define $\varphi_D^j(s):=g_D(s)-h_D^j(s)$ for $j=a,b$. We have:
$$
\varphi_D^j(s)=\frac{D(s_j(D)-s)}{k\mu(s)} +\frac{\mu(s)-D}{k\mu'(s)},$$ 
$$
{\varphi^j_D}'(s)=\frac{\mu(s)-D}{k}\left(\frac{1}{\mu(s)}-\frac{\mu''(s)}{\mu'(s)^2}\right) - (s_j(D)-s)\frac{D\mu'(s)}{k\mu(s)^2}.$$
First, we consider $\varphi_D^a(s)$ on $(0,\bar s)$. Given that $\mu(s)$ is increasing and concave on this interval, we get ${\varphi^a_D}'(s)<0$ on $(0,s_a(D))$, and ${\varphi^a_D}'(s)>0$ on $(s_a(D),\bar s)$. Moreover, we have ${\varphi^a_D}(s_a(D))=0$, so $\varphi_D^a(s)\geq0$ on $(0,\bar s)$, which proves (i).\\ 

 Now we want to determine the sign of $\varphi_D^b(s)$ on $(\bar s,+\infty)$.
For this purpose, we consider the equation $\varphi_D^b(s)=0$. By replacing $\mu(s)$ and its derivative by their analytic expressions, this equation becomes:
$$
\frac{s_a(D)}{k_I}s^2 -2k_S s +k_S s_b(D)=0.
$$
Given that $s_a(D)s_b(D)=k_S k_I$, the equation $\varphi_D^b(s)=0$ has only one root $s=s_b(D)$. Moreover, we have:
$$
\lim_{s\searrow\bar s}\varphi_D^b(s)=-\infty \quad \textrm{and} \quad \lim_{s\rightarrow+\infty}\varphi_D^b(s)=-\infty.
$$
Given that $\varphi_D^b(s)$ is continuous on $(\bar s,+\infty)$, we finally conclude that on this interval, $\varphi_D^b(s)\leq0$ , i.e. $g_D(s)\leq h^b_D(s)$. \qed

\subsection*{Proof of Lemma \ref{lem-doty}}

First, given that $g_D(s)$  represent the nullcline $\dot y(\xi)=0$, we can check that we have  $\dot y(\xi)>0$ on (see Figure \ref{fig:Dcst}):
\begin{align*}
\{\xi \in int\r \mid s<\bar s, x<g_D(s) \} \\ 
& \cup\{\xi \in int\r \mid s>\bar s, x>g_D(s) \}.
\end{align*}

Recalling that $h^a_D(s)$ and $h^b_D(s)$ are respectively the isolines $y(\xi)=y_a(D)$ and $y(\xi)=y_b(D)$, Lemma \ref{lem-y} allows to conclude that for $\xi\in \r$ such that $y_b(D) <y(\xi)<y_a(D)$, we have $\dot y(\xi)>0$. \qed

\end{document}